\documentclass[a4paper,11pt]{article}
\usepackage{mathtools}
\usepackage{amsmath}
\usepackage{mathrsfs}
\usepackage{amssymb}
\usepackage{amsthm}
\usepackage{dsfont}
\usepackage{graphicx}
\usepackage{bmpsize}
\usepackage{times}
\usepackage{textcomp }
\usepackage{mathtools}
\usepackage{latexsym, empheq, fancybox}
\usepackage{mathrsfs}
\usepackage{exscale}
\usepackage{enumerate}
\usepackage{booktabs, array}
\usepackage[english]{babel}
\newtheorem{theorem}{Theorem}[section]
\newtheorem{lemma}[theorem]{Lemma}

\newtheorem{corollary}[theorem]{Corollary}

\newtheorem{prop}{Proposition}[section]
\usepackage{authblk}

\begin{document}
\title{ Bowen entropy of sets of  generic points for fixed-point free flows
       \footnotetext {*Corresponding author}
		}

\author{Yunping Wang$^1$, Ercai Chen$^{1*}$ , Ting Wu$^1$, Zijie Lin$^{1}$\\
\small1    School of Mathematical Sciences and Institute of Mathematics, Nanjing Normal University,\\
\small   Nanjing 210046, Jiangsu, P.R.China\\

\small    yunpingwangj@126.com,
 ecchen@njnu.edu.cn,
 swuting@126.com, zilin137@126.com
}

\date{}

\maketitle{}
\begin{abstract}
Let $(X, \phi)$ be a compact metric flow without fixed points. We will be concerned with the entropy of flows which  takes into consideration all possible reparametrizations of the flows.
In this paper,  by  establishing  the Brin-Katok's entropy  formula  for flows without fixed  points in non-ergodic case, we prove the  following result: for an ergodic $\phi$-invariant measure $\mu$,
	$$ h_{top}^{B}(\phi, G_{\mu}(\phi))=h_{\mu}(\phi_{1}),$$
	where $G_{\mu}(\phi)$ is the set of generic points for $\mu$ and $h_{top}^{B}(\phi, G_{\mu}(\phi))$ is the Bowen entropy  on $G_{\mu}(\phi)$.  This extends the classical result of  Bowen in 1973 to fixed-point free flows. Moreover, we show that the Bowen entropy can be determined via the local entropies of measures.
\end{abstract}

\section{Introduction}
Throughout  the paper by a flow we mean a pair $(X, \phi)$, where $(X, d)$ is a compact metric  space with metric $d$,   and $\phi:X\times\mathbb{R}\to X$ a continuous flow on $X$, that is , $\phi_{t}:X\rightarrow X$ is a homeomorphism given by $\phi_{t}(x)=\phi(x,t)$ for each $t\in \mathbb{R}$ and satisfies $\phi_{t}\circ\phi_{s}=\phi_{s+t}$ for each $t,s \in \mathbb{R}$.
 A Borel probability measure $\mu$ on $X$ is called $\phi$-invariant if for any Borel set $B$, it holds $\mu(\phi_{t}(B)) = \mu(B)$ for all $t\in \mathbb{R}$. It is called ergodic if any $\phi$-invariant Borel set has measure $0$ or $1$. The set of all Borel probability measures, all $\phi$-invariant Borel probability measures and all ergodic $\phi$-invariant Borel probability measures on $X$ are denoted by $\mathcal{M}(X)$, $\mathcal{M}_{\phi}(X)$ and $\mathcal{E}_{\phi }(X)$, respectively.

The notion of entropy plays a crucial role  in quantifying  the degree  of "disorder" in the systems. For a flow,  it is sometimes useful to represent measure-theoretic entropy of time one map by the whole flow itself. However, an invariant measure for time one map is not, in general, flow invariant and similarly, an ergodic measures for a flow is not necessarily ergodic for time one map. Thus there exist significant non-parallel gradients  between flows and its discrete sample. There are some fruitful  works devoting to investigate a proper definition of the entropy of a flow.
In  \cite{Abramov}, Abramov constructed   Abramov entropy formula, that is , $h_{\mu}(\phi_{t}) = |t|h_{\mu}(\phi_{1})$ for all $t\in \mathbb{R}$ which reveals the basic relationship between measure-theoretic entropy of time one map and measure-theoretic entropy of flows. 
   Bowen \cite{Bow2} gave the definition of topological entropy for one parameter flows on compact metric spaces and  proved that defined topological  entropy  of a flow  is  equivalent to  the  topological entropy  of time one map.  Sun and Vergas defined  both measure-theoretic entropy  and topological entropy in \cite{Sun2} and proved that so defined measure-theoretic entropy  and topological entropy were both equal to that of time one map. Shen and Zhao \cite{Zhao} established the  variational principle between topological entropy and measure-theoretic entropy of a flow and constructed the Brin-Katok's entropy formula for a flow. 
   
   All of the above results were studied in the case of the usual Bowen ball. It was defined by  $$B_{t}(x , \epsilon, \phi)=\left\lbrace y\in X: d(\phi_{s} x, \phi_{s}y)< \epsilon, \forall ~ 0\leq s \leq t\right\rbrace.$$ 
   However,  in this paper, we consider   reparametrization balls. For convenience, we review some of  the standard  facts on  reparametrizations.
   Let $I$ be a closed interval  which contains the origin, a continuous map $\alpha: I\to \mathbb R $ is said to be a reparametrization if it is a homeomorphism onto its image and $\alpha (0)=0.$ Define
   $Rep(I)$ to be the set of all  reparametrizations on $I$.  For a flow $\phi $ on $X$, given $x\in X, t\geq 0$ and $\epsilon >0$, we put
   $$B(x, t, \epsilon ,\phi )=\{y\in X: \exists~\alpha \in Rep[0, t]~\text{s.t.}~d(\phi_{\alpha(s)}x, \phi_s y)<\epsilon, \forall~0\leq s\leq t\},$$and call such set a $(t, \epsilon ,\phi)$-ball or a reparametrization ball. Clearly, all the reparametrization balls are open sets. To investigate the topological entropies of mutually conjugate expansive flows, Thomas \cite{Thoma1} introduced the entropy for flows which raised from allowing reparametrizations of orbits. Subsequently, he showed  that his definition of entropy was equivalent to Bowen's definition for any flow without fixed points on compact metric spaces in \cite{Thoma2}. 
    
 In recent years, reparametrization balls attract a lot of attention. In \cite{Sun1}, following  the ideas  of Katok's entropy formula,  they  defined the measure-theoretic entropy of  a flow by using  the reparametrization balls and showed that so defined measure-theoretic entropy was equal to that of time one map when the measure is ergodic.  
    Recently, Dou etc. \cite{Dou}  introduced  Bowen entropy  for compact metric flows through  reparametrization balls and  established a variational principle which generalized the result in \cite{Huang}. Meanwhile,  they defined lower and upper measure-theoretic entropy for any Borel probability measure $\mu$,
    $$\underline{h}_{\mu}(\phi)=\int \underline{h}_{\mu}(\phi, x) d\mu  ~\text{and}~ \overline{h}_{\mu}(\phi)=\int \overline{h}_{\mu}(\phi, x) d\mu,$$ where
    $$ \underline{h}_{\mu}(\phi, x)=\lim\limits_{\epsilon \to 0} \liminf \limits_{t\to \infty } -\frac{1}{t} \log \mu(B(x, t, \epsilon, \phi ))$$
    and
    $$ \overline{h}_{\mu}(\phi, x)=\lim\limits_{\epsilon \to 0} \limsup \limits_{t\to \infty } -\frac{1}{t} \log \mu(B(x, t, \epsilon, \phi )).$$ 
   They  raised a question whether  $\underline{h}_{\mu}(\phi)=\underline{h}_{\mu}(\phi_{1})$ and $\overline{h}_{\mu}(\phi)=\overline{h}_{\mu}(\phi_{1})$  hold for every Borel probability  measure $\mu$.
Recently, we \cite{YJ}  established  the Brin-Katok's formula for compact metric flows without fixed points in ergodic case, which partially  gives   a positive  answer to the above question. Motivated by this work, we prove the Brin-Katok's formula for compact metric flows without fixed points in non-ergodic case which is useful for the proof of Theorem \ref{main}. Let $C(X)$ denote the space of real-value continuous functions of $X$ equipped with the supremum  norm. The terms $\mathcal{M}_{\phi_{1}}(X)$ and $\mathcal{E}_{\phi_{1}}(X)$ represent the sets of all $\phi_{1}$-invariant Borel probability measures and ergodic $\phi_{1}$-invariant  Borel probability measures,  respectively. 
For $ \mu \in \mathcal {M}_{\phi_{1}}(X),$
put
$$ G_{\mu}=\{x\in X: \lim\limits_{n\rightarrow \infty}\frac{1}{n}\sum\limits_{k=0}^{n-1}f(\phi_{1}^{k}x)= \int_{X} f d \mu,  ~~ \forall f \in  C(X)\}$$
be the set of generic points of $\mu$. By Birkhoff's Ergodic Theorem and Ergodic Decomposition Theorem, we have $\mu( G_{\mu})=1$ if $\mu$ is  ergodic for  $\phi_{1}$ and
$\mu( G_{\mu}) = 0$ if $\mu$ is not   ergodic for $\phi_{1}$.  Bowen \cite{Bow1} defined a kind of topological entropy with
characteristic of dimension type which is so-called Bowen entropy and proved the following remarkable result.
\begin{theorem}\rm\cite{Bow1}
	Let $(X, d)$  be a  compact metric space, $\phi_{1}: X \rightarrow X$ be a continuous map and $\mu\in \mathcal{E}_{\phi_{1}}(X)$. Then
	$$ h_{top}^{B}(\phi_{1}, G_{\mu})=h_{\mu}(\phi_{1}),$$
	where $h_{top}^{B}(\phi_{1},G_{\mu} ) $ is the Bowen entropy of $G_{\mu}$ for time-one map.
\end{theorem}
 It is natural  to ask: Does above the result also hold for compact metric flows without fixed points? 
In this paper, by using different approach of  Bowen's original proofs for $\mathbb{Z}$-actions,  we proved the following theorems:
\begin{theorem}\label{main2}
	Let $(X, \phi)$ be a compact metric flow without fixed points. For $\mu \in \mathcal{M}_{\phi}(X)$, if 
	 ~$Y \subset X $ and $\mu(Y)=1$, then $h_{\mu}(\phi_{1}) \leq h_{top}^{B}(\phi, Y).$
\end{theorem}

\begin{theorem}\label{main}
	Let $(X,\phi)$ be a compact metric flow without fixed points and $\mu\in \mathcal{E}_{\phi}(X)$. Let
	\begin{equation}
	G_{\mu}(\phi)=\{ x\in X :\lim_{t\rightarrow \infty}\dfrac{1}{t}\int_{0}^{t}f(\phi_{\tau}x)d\tau =\int_{X} f d\mu, ~~~\forall f \in C(X) \}
	\end{equation}
	be the set of generic points for $\mu$, then $$h_{top}^{B}(\phi, G_{\mu}(\phi))=h_{\mu}(\phi_{1}).$$	
\end{theorem}
It is worthy mentioning that since the Birkhoff's Ergodic Theorem for flows (see \cite{Kry}), $\mu(G_{\mu}(\phi))=1$ if $\mu \in \mathcal{E}_{\phi}(X)$. $G_{\mu}(\phi)$  may be an empty  set when  $\mu$ is not ergodic.
Then we give  a lower bound for $h_{top}^{B}(\phi, G_{\mu}(\phi))$ by Theorem 
\ref{main2}. We also note that  for the proof of  Theorem \ref{main2}, we use a non-ergodic version of  Brin-Katok's entropy formula (Theorem \ref{ka}) and a variational principle for Bowen entropy \cite{Dou}.  For the upper bound, we apply the ideas of Pfister and Sullivan \cite{Pfi} to prove Theorem \ref{main}. The idea of proof is inspired by Zheng and Chen \cite{Zheng}.
Moreover, we show that the Bowen entropy can be determined by  the local entropies of measures. This result can be considered as an analogue of Billingsley's Theorem for the Hausdorff dimension \cite{Bi}. The key to proof of Theorem \ref{main3} is that we need to  overcome technical difficulties arising from allowing reparametrizations of orbits.
\begin{theorem}\label{main3}
	Let $(X,\phi)$ be a compact metric flow without fixed points. For any $\mu \in  \mathcal{E}_{\phi}(X)$, $E$ be a Borel subset of $X$ and  
	$0<s<\infty$.
	\begin{enumerate}
		\item If $\underline{h}_{\mu}(x)\leq s $ for all $x\in E$, then $h^B_{top}(\phi, E)\leq s$.
		\item If $\underline{h}_{\mu}(x)\geq s$ for all $x\in E$ and $\mu (E)>0$, then $h_{top}^B(\phi, E)\geq s$.
	\end{enumerate}
\end{theorem}

The remainder of this paper is organized as follows. In Section \ref{two}, we introduce  Bowen entropy for flows and present basic concepts  concerning the measure-theoretic entropy for time one map. In Section \ref{three}, we prove the Brin-Katok's entropy formula for non-ergodic case. The proofs of Theroem \ref{main2}, Theorem \ref{main}, Theorem \ref{main3}  are given in  Section \ref{four}, Section \ref{five}, Section \ref{six}, respectively.
\section{Preliminaries}\label{two}
\subsection{Bowen entropy for compact metric flows }

In this subsection, we first introduce Bowen entropy for compact  metric flows  \cite{Dou}.
 Let $(X,\phi)$ be a flow and $Z\subset X$. For $s\geq0$, $N\in\mathbb{N}$ and $\epsilon>0$, set $$\mathcal{M}_{N,\epsilon}^s(\phi,Z)=\inf\sum_{i}\exp(-st_i),$$ where the infimum is taken over all finite or countable families of reparametrization balls $\{B(x_i,t_i,\epsilon,\phi) \}$, $x_i\in X$ and $t_i\geq N$ such that $Z\subset\cup B(x_i,t_i,\epsilon,\phi)$. Then the following limits exist: $$\mathcal{M}_{\epsilon}^s(\phi,Z)=\lim_{N\to\infty}\mathcal{M}_{N,\epsilon}^s(\phi,Z),\ \mathcal{M}^s(\phi,Z)=\lim_{\epsilon\to0}\mathcal{M}_{\epsilon}^s(\phi,Z).$$
The Bowen entropy $h_{top}^B(\phi,Z)$ is defined as critical value of the parameters, where $\mathcal{M}^s(\phi,Z)$ jumps from $\infty$ to $0$, i.e.,
\begin{align*}
	h_{top}^B(\phi,Z)&=\inf\{s:\mathcal{M}^s(\phi,Z)=0 \}\\
	&=\sup\{s:\mathcal{M}^s(\phi,Z)=\infty \}.
\end{align*}
\begin{prop}\label{w} \rm\cite{Dou}
$ $	
\begin{enumerate}[1.]
\item If $Z_{1}\subset Z_{2}\subset X$, then $h_{top}^B(\phi,Z_{1})\leq h_{top}^B(\phi,Z_{2})$.
\item If $Z_{i}\subset X$ for $i=1,2,...,$ then $h_{top}^B(\phi,\bigcup_{i=1}^{\infty}Z_{i})= \sup\limits_{i\geq 1}h_{top}^B(\phi,Z_{i})$.
\end{enumerate}
\end{prop}
\subsection{Measure-theoretic entropy for time one map }

In this subsection, we recall some notations and result on the  measure-theoretic entropy for time one map $\phi_{1}$. 
Note that $\phi_{1} : X \rightarrow X$ is a continuous map on the compact metric space $X$ with metric $d$. Let $ \mathcal{M}_{\phi_{1}}(X), \mathcal{E}_{\phi_{1}}(X)$ denote the sets of all $\phi_{1}$-invariant Borel probability  measures and $\phi_{1}$-invariant 
ergodic  Borel probability measures, respectively. Let $\mathcal{B}(X)$ be  the Borel $\sigma$-algebra of $X$. A partition of $X$ is a disjoint collection of elements of $\mathcal{B}(X)$ whose union is $X$. Let $\mathcal{P}(X)$ denote the collection of all finite measurable partitions of $X$. Given two partitions $\alpha, \beta$ of $X$, $\alpha$ is said to  be finer than $\beta$ (denoted by $\alpha \succeq \beta$) if each element of $\alpha$ is contained in some element of $\beta$. Let $\alpha \vee \beta= \left\lbrace A\bigcap B: A\in \alpha, B\in \beta \right\rbrace .$

 For $\xi\in \mathcal{P}(X)$  and $x\in X$, denote  by $\xi(x)$ the element of $\xi$ containing $x$, and set $\xi_{n}=\xi\vee  \phi_{1}^{-1} \xi \vee \cdots \vee \phi_{1}^{-(n-1)}\xi $.
 The each $\phi_{1}$-invariant measure $\mu$ induces a measure preserving dynamical system $(X, \mathcal{B}(X), \mu, \phi_{1} )$. Consider  the $\sigma$-algebra $\mathcal{T}_{\mu}=\{A\in \mathcal{B}(X): \mu(A\bigtriangleup \phi_{1}^{-1} A)=0\}$.
Let $\rho: X \rightarrow X/\mathcal{T}_{u}:= Y$ be the associated projection and $\mu=\int_{Y} \mu_{y} d \pi(y)$ be the  decomposition of $\mu$ over $Y$. Such a decomposition  is called the \emph{ergodic decomposition} of $\mu$, since for each $y\in Y$, $\rho^{-1}(y)$ is $\phi_{1}$-invariant and $(\rho^{-1}(y), \phi_{1}, \mu_{y})$ is $\phi_{1}$-ergodic measurable dynamical system.  
The following  is the non-ergodic version of Shannon-McMillan-Breiman theorem \cite{WP}.
\begin{theorem}\rm\cite{WP}
	Let $(X,\mathcal{B}({X}), \mu, \phi_{1})$ be a  measure preserving  dynamical system. Then for any   $\xi\in  \mathcal{P}(X)$ one has that for $\mu$-a.e.$x\in X$, 
	$$ \lim\limits_{n\rightarrow \infty}-\dfrac{1}{n}\log \mu(\xi_{n}(x))=h_{\mu_{y}}(T, \xi|\rho^{-1}(y))\triangleq h(x, \xi),$$
	where $y\in Y$ such that $\rho^{-1}(y)$ is  the ergodic component containing $x$ and
	$$ \int_{X} h(x, \xi) d \mu(x)= h_{\mu}(\phi_{1}, \xi).$$
\end{theorem}
\section{Brin-Katok's entropy formula for  non-ergodic case}\label{three}

In this section, we will prove Brin-Katok's entropy formula  for  compact metric flows. The statement of this formula is the following.
\begin{theorem}\label{ka}
	Let $(X, \phi)$ be a compact metric flow without fixed points. For any $\mu \in \mathcal {M}_{\phi}(X)$, then 
$$\underline{h}_{\mu}(\phi)=\overline{h}_{\mu}(\phi)=h_{\mu}(\phi_{1})$$	
	and for $\mu$-a.e. $x\in X,$  $ \overline h_{\mu}(\phi , x)= \underline h_{\mu}(\phi, x)$.
\end{theorem}

 Theorem \ref{ka} can be obtained form the following Proposition \ref{m1}
and \ref{m2}. 
\begin{prop}\label{m1}
Let $(X, \phi)$ be a compact metric flow. For any $\mu \in \mathcal{M}_{\phi }(X)$, then 
		$$ \int _{X} \overline{h}_{\mu}(\phi, x) d \mu \leq \frac{1}{|\tau |}h_{\mu}(\phi_\tau), $$
		for all $\tau \in \mathbb R\setminus \{0\}.$
\end{prop}
\begin{proof}	
{\bf Case 1.} Consider $\tau >0.$

\noindent Notice that 
$$\limsup\limits_{n\to \infty } -\dfrac{\log \mu ( B(x, n\tau, \epsilon, \phi ))}{n\tau }= \limsup\limits_{t\to \infty } -\dfrac{\log \mu ( B(x,t , \epsilon, \phi ))}{t}.$$
 Indeed, for $t>0$,  choose  $n_t\in \mathbb N$ such that $n_t\tau \leq t< (n_t +1)\tau. $
 Then we have
 $$B(x, (n_t+1)\tau , \epsilon, \phi )\subset B(x, t, \epsilon, \phi )\subset B(x, n_t \tau , \epsilon, \phi ).$$
 Therefore, 
\begin{align*}
\limsup\limits_{t\to \infty } -\dfrac{\log \mu ( B(x,t, \varepsilon, \phi ))}{t }
&\leq  \limsup\limits_{t\to \infty } -\dfrac{\log \mu ( B(x,(n_t  +1)\tau , \varepsilon, \phi ))}{t }\\
&\leq \limsup\limits_{t\to \infty } -\dfrac{\log \mu ( B(x,(n_t  +1)\tau , \varepsilon, \phi ))}{n_t\tau }\\
&=\limsup\limits_{t\to \infty } -\dfrac{\log \mu ( B(x,(n_t  +1)\tau , \varepsilon, \phi ))}{(n_t +1)\tau }.
\end{align*}
Then it is enough to prove the result for $t=n\tau, ~n \in \mathbb{N}.$

For any $\epsilon >0$, choose    $\eta >0$  such that $d(\phi_sx, \phi_s y)<\epsilon $, $\forall s\in[0,\tau] $ if $d(x,y)<\eta$.
		For any $x\in X$, we define
		$${D}(x, n, \eta , \phi_{\tau } )=\{y \in X: d(\phi_{i\tau } x, \phi_{i \tau } y)<\eta,  i=0, 1, \cdots, n-1 \},$$
		$$B_{t}(x,\epsilon, \phi )=\{y\in X: d(\phi_sx,\phi_sy)<\epsilon, ~ \forall~ 0\leq s\leq t \}.$$
Then
		$${D}(x, n, \eta , \phi_{\tau } )\subset B_{t}(x,\epsilon, \phi )\subset B(x, t, \epsilon, \phi ).$$
	For a finite measurable partition  $\beta$, let ${\rm diam}(\beta)$= $\max\left\lbrace {{\rm diam}(A): A\in \beta}\right\rbrace .$	Choose a finite measurable partition $\xi $ of $X$ with ${\rm diam }(\xi )<\frac{\eta }{2}$.   Then by  SMB theorem, for $\mu$-a.e. $x \in X,$
		$$  \int_{X} \lim\limits_{n\to \infty } -\dfrac{\log \mu (\xi_n(x))}{n} d\mu =h_{\mu }(\phi_{\tau} , \xi  )\leq h_{\mu }(\phi_{\tau } ),$$
		where $\xi_n=\xi\vee  \phi_{\tau }^{-1} \xi \vee \cdots \vee \phi_{\tau}^{-(n-1)}\xi$ and $\xi_{n}(x)$ be the element of $\xi_{n}$ containing $x$.
		Since $\xi_n(x)\subset {D}(x, n, \eta , \phi_{\tau } ) \subset B(x, n\tau, \epsilon, \phi ),$
		$$\int_{X} \lim\limits_{\epsilon \to 0}\limsup\limits_{n\to \infty } -\dfrac{\log \mu ( B(x, n\tau, \epsilon, \phi ))}{n } d\mu \leq h_{\mu}(\phi_{\tau}).$$
	It follows that 
		$$ \int_{X}\overline{h}_{\mu}(\phi, x) d\mu \leq \frac{1}{ \tau  }h_{\mu}(\phi_\tau).$$

\noindent{\bf Case 2.} Consider $\tau <0.$
		
		Then we have $-\tau >0,$ by {\bf Case 1}, we obtain 
		$$ \int_{X}\overline{h}_{\mu}(\phi, x)d\mu \leq -\frac{1}{\tau } h_{\mu }(\phi_{-\tau }) =-\frac{1}{\tau } h_{\mu }(\phi_{\tau})=\frac{1}{|\tau |}h_{\mu }(\phi_{\tau }). $$

\end{proof}

\begin{lemma}\label{imp}\rm\cite{Thoma1}	Let $(X, \phi)$ be a compact metric flow without fixed points. For any $\eta >0,$ there exists $\theta >0$ such that for any $x, y\in X$, and any closed interval $I$ containing the orign, and any reparametrization $\alpha \in Rep(I), $ if $d(\phi_{\alpha(s)}(x), \phi_s(y))< \theta $ for all $s\in I$, then it holds that
	$$|\alpha(s)-s|<\left\{
	\begin{array}{ll}
	\eta  |s|, & \hbox{if $|s|>1$;} \\
	\eta, & \hbox{if $|s|\leq 1$.}
	\end{array}
	\right.$$
\end{lemma}

\begin{prop}\label{m2}
		Let $(X, \phi)$ be a compact metric flow without fixed points. For any $\mu \in \mathcal{M}_{\phi }(X)$, then
		$$ \int_{X}\underline{h}_{\mu}(\phi, x)d\mu \geq \frac{1}{|\tau |}h_{\mu}(\phi_\tau), $$
		for all $\tau \in \mathbb R\setminus \{0\}.$
\end{prop}

\begin{proof}
Fix $\tau>0$, without loss of generality  we may assume that $h_{\mu}(\phi_{\tau})>0$. For any $p>0$, we will show
$$p+ \int_{X}\underline{h}_{\mu}(\phi, x) d\mu \geq \frac{1}{\tau }h_{\mu}(\phi_\tau).$$
We choose $L\in \mathbb N$ such that $L\geq\dfrac{2\log6+p}{p\tau}$.
		We divide the proof into the following two cases.

\noindent{\bf Case 1.} Consider $\tau >0.$ \\
  Similar to the above argument, we have  
  $$\liminf\limits_{t\to \infty } -\dfrac{\log \mu ( B(x,t, \epsilon, \phi ))}{t }=\liminf\limits_{n\to \infty } -\dfrac{\log \mu ( B(x,nL\tau , \varepsilon, \phi ))}{nL \tau }.$$

\noindent Then it is sufficient to prove the case of   $t=nL\tau , n\in \mathbb N.$

   Consider the $\sigma$-algebra $\mathcal{A}_{\mu}=\{A\in \mathcal{B}(X): \mu(A\bigtriangleup \phi_{-L \tau}A)=0\}$.
 Let $\rho: X \rightarrow X/\mathcal{A}_{u}:= Y$ be the associated projection and $\mu=\int_{Y} \mu_{y} d \pi(y)$ be the $\phi_{L\tau}$-ergodic decomposition of $\mu$.  For each  $y\in Y$, $\rho^{-1}(y)$ is $\phi_{L\tau}$-invariant and $(\rho^{-1}(y),\phi_{L\tau},\mu_y)$ is a ergodic dynamical system. Let $h(y)=h_{\mu_{y}}(\rho^{-1}(y), \phi_{L\tau})$ be the measure theoretic entropy  restricted to the system $(\rho^{-1}(y), \phi_{L\tau}, \mu_{y})$. For any $M >0$, set
$X_{M}=\rho^{-1}(h^{-1}([0, M)))$, $ X_{M}^{'}=\rho^{-1}(h^{-1}([M, \infty)))$ and $X_{\infty}=\rho^{-1}(h^{-1}(\infty)).$  Then $ X= X_{M} \bigcup X_{M}^{'}\bigcup X_{\infty}.$
Therefore, we can have the following relationship which plays a crucial role in our proof.

\noindent{\bf Claim 1.}\label{m}

(1) For any $M>0$,
\begin{align}\label{1}
\int_{X_M}  p+ \lim\limits_{\epsilon \rightarrow 0}\liminf\limits_{n\rightarrow \infty} -\frac{1}{nL\tau}\log \mu(B(x,nL\tau,\epsilon,\phi))d\mu \geq \frac{\int_{h^{-1}([0, M))} h(y) d \pi(y)}{L\tau}.
\end{align}

(2) For $\mu$ almost every $x\in X_{\infty}$,
\begin{align}\label{2}
\lim\limits_{\epsilon \rightarrow 0}\liminf\limits_{n\rightarrow \infty} -\frac{1}{nL\tau}\log \mu(B(x,nL\tau,\epsilon,\phi)) = \infty.
 \end{align}

\noindent The  proof of Claim 1    will be given later.
 Now we proceed with the proof of Proposition{ \ref{m2}}.

According to Claim 1, let 
 $M$ tend to $\infty$, then $\mu(X_{M}\bigcup X_{\infty})$ tend to 1. Hence 
\begin{align*}
\int_{X} p+\underline{h}_{\mu}(\phi, x) d\mu \geq \frac{\int_{Y}h_{\mu_{y}}(\phi_{L \tau},\rho^{-1}(y))d\pi(y)}{L\tau} 
=\frac{h_{\mu}(\phi_{L\tau})}{L\tau}=\frac{h_{\mu}(\phi_{\tau})}{\tau}.
\end{align*}
 As $p>0$ was chosen arbitrary, we conclude that 
\begin{align*}
\int_{X} \underline{h}_{\mu}(\phi, x) d\mu \geq \frac{h_{\mu}(\phi_{\tau})}{\tau}.
\end{align*}


\noindent{\bf Case 2.} Consider $\tau <0.$
Then we have $-\tau >0,$ by {\bf Case 1} we have
$$  \int_{X} \underline{h}_{{\mu}}(\phi, x)d \mu \geq -\frac{1}{\tau } h_{\mu }(\phi_{-\tau })=-\frac{1}{\tau } h_{\mu }(\phi_{\tau})=\frac{1}{|\tau |}h_{\mu }(\phi_{\tau }). $$


\noindent This completes the proof of Proposition \ref{m2}, modulo the Claim 1.

\end{proof}

\begin{proof}[Proof of Claim 1]
	If $\mu(X_{M})= 0$ and $\mu(X_{\infty})=0$, then (\ref{1}) and (\ref{2}) hold respectively.  So we may assume that both $\mu(X_{M})$ and $\mu(X_{\infty})$ are positive.

	Take $K\in \mathbb{N}$ to be sufficiently large and let $\gamma = \frac{M}{K}.$ For $k= 0, 1, \cdot \cdot \cdot, K-1$, let $B_{k}= \rho^{-1}(h^{-1}[k\gamma, (k+1)\gamma))$ and  $B_{\infty}= X_{\infty}.$

	Let $\xi^{m}=\{ A_{1}^{m}, A_{2}^{m}, \cdot \cdot \cdot , A_{c_{m}}^{m} \}$ be a sequence of finite measurable partitions of $X$  such that for each $m$
	\begin{enumerate}[1.]
		\item $A_1^{m}, \cdots, A_{c_{m}-1}^{m}$ are piecewise disjoint compact sets;
		\item $A_{c_{m}} =X\setminus \bigcup\limits_{i=1}^{c_{m}-1} A_i$;
		\item $\lim\limits_{m\rightarrow \infty}\rm{diam}(\xi^{m})=0.$
	\end{enumerate}
	Then
	$$ \lim\limits_{m\rightarrow \infty} h_{\nu}(\phi_{L\tau},\xi^{m})= h_{\nu}( X, \phi_{L \tau}), ~~\text {for}~\text{any}~~ \nu \in \mathcal{M}_{\phi_{L\tau}}(X).$$
	The element of  $\bigvee_{i=0}^{n-1} \phi_{L\tau}^{-i} \xi^{m}$ which contains $x$ will be denoted by $\xi_{n}^{m}$.
	By the SMB theorem, for $\mu$-a.e. $x\in X$,
	$$ \lim\limits_{n\rightarrow \infty} -\frac{1}{n}\log \mu (\xi_{n}^{m}(x))\triangleq h(x, \xi^{m})= h_{\mu_{y}}(\phi_{L\tau},\xi^{m}| \rho^{-1}(y))= h_{\mu_{y}}(\phi_{L \tau}, \xi^{m}), $$
	where $\rho^{-1}(y)$ is the ergodic component that contain $x$, i.e., $\rho(x)=y$.  Hence for $\mu$-a.e. $x\in X, \lim\limits_{m\rightarrow \infty} h(x, \xi^{m})= h_{\mu_{y}}( \rho^{-1}(y),\phi_{L\tau})= h(y),$ where $y=\rho(x).$
	
	For any $0< b < \min\{\gamma, p\}$, by Egorov's theorem, we then can choose $\xi=\xi^{m}$ for $m$ sufficiently large such that up to a subset of $X$ with small $\mu$ measure (say, less than $b$), it holds that $h(x, \xi)> \min\{ \frac{1}{b}, h(\rho(x))-b\}.$ Hence there exists sufficiently large $N$, whence $n > N$, 
	$\mu(E_{k})> \mu(B_{k})-2b$ for each $k= 0, 1, \cdot \cdot \cdot, K-1,$
	and
	$\mu(E_{\infty})> \mu(B_{\infty})- 2b,$
	where 
	$$E_{k}=\{x\in B_{k}:\forall n^{'}\geq n, -\frac{1}{n}\log \mu(\xi_{n} (x))> k\gamma- 2b  \}$$
	and 
	$$E_{\infty}=\{x\in B_{\infty}: \forall n^{'}\geq n, -\frac{1}{n}\log \mu(\xi_{n}(x))> \frac{1}{b}-2b\}$$
Suppose that  $\xi=\{ A_{1}, A_{2},\cdot\cdot\cdot, A_{q}, A_{q+1}\}$. Let $\eta_0=\min\{ d(A_{i}, A_{j}): 1\leq i\neq j\leq q\}$. Fix  $\eta \in (0, \eta_0)$, we choose $\theta >0$
	such that
	\begin{align*}
	d(\phi_s(z), z)<\frac{\eta}{3}
	\end{align*}
	for all $z\in X, |s|\leq \theta.$ By Lemma \ref{imp},  for $\eta=\dfrac{\theta }{ 4L\tau }$
	, we choose $\epsilon \in (0, \dfrac{\eta }{3})$.
	For any $x\in X$, we define
	$$W_n:=\left \{ A\in \bigvee_{i=0}^{n-1} \phi_{L\tau }^{-i} \xi:~ A\cap B(x, t, \epsilon, \phi ) \neq \emptyset \right \}.$$
	Next we will estimate the numbers of $ W_{n}$.
	Notice that
	\begin{align*}
	B(x, t, \epsilon, \phi )\subset \bigcup_{A\in W_n} A.
	\end{align*}
	For $y \in A\cap B(x, t, \epsilon, \phi), A\in \bigvee_{i=0}^{n-1} \phi^{-i}_{L\tau} \xi $, there exists $\alpha \in Rep([0, t])$, such that
	$$d(\phi_{\alpha(s)}x , \phi_s y)<\epsilon, ~0\leq s\leq t. $$ 
	 Fix $s_1\in[0,t]$, set $u:=s-s_1, \gamma(u):=\alpha(s)-\alpha(s_1)$, then $\gamma \in Rep([-s_1, t-s_1]),$ satisfying
	$$d(\phi_{\gamma (u)}\phi_{\alpha(s_1) }x, \phi_u \phi_{s_1} y)<\epsilon, ~-s_1\leq u\leq t-s_1.$$
	By Lemma \ref{imp}, one obtains $$|\gamma(u)-u|<\left\{
	\begin{array}{ll}
	\varepsilon_1  |u|=\dfrac{\theta}{4L\tau}|u|, & \hbox{if $|u|>1$;} \\
	\varepsilon_1=\dfrac{\theta}{4L\tau}<\dfrac{\theta}{4}, & \hbox{if $|u|\leq 1$.}
	\end{array}\right.$$
	For $s_2\in[-s_1,t-s_1]$ satisfying $|s_1-s_2|<L\tau$,  we have
	$$|(\alpha (s_1)-s_1)-(\alpha(s_2)- s_2)|\leq \frac{\theta }{4}.$$
    We consider the following sequence:
	$$S_\alpha= \left\{  \lfloor \dfrac{\alpha(kL\tau )-kL\tau }{\theta / 4} \rfloor  \right\}, ~k=0, 1, \cdots, n-1,$$
	where $\lfloor z\rfloor$ denotes the largest integer  less or equal $z$.
	If for some $ \widetilde{A}\in W_{n}$ with $ \widetilde{A}\neq A$, there exists $z\in  \widetilde{A}\cap B(x, t, \epsilon, \phi )$, we can choose $\beta \in Rep[0, t]$ such that $d(\phi_{\beta(s)}x , \phi_s z)<\epsilon, ~0\leq s\leq t.$ Similarly, we can define the sequence $S_{\beta }$. If $S_{\alpha}=S_{\beta }$, for any $s\in [0, t]$, we have
	\begin{align*}
	\begin{split}
	|\alpha(s)-\beta(s)|&\leq \left| (\alpha(s)-s)-\big(   \alpha \big(  \lfloor\frac{s}{L\tau}\rfloor L\tau   \big)- \lfloor\frac{s}{L\tau}\rfloor L\tau \big)    \right|\\
	+&\left|\big(   \alpha \big(  \lfloor\frac{s}{L\tau}\rfloor L\tau   \big)- \lfloor\frac{s}{L\tau}\rfloor L\tau \big) - \big(   \beta \big(  \lfloor\frac{s}{L\tau}\rfloor L\tau   \big)- \lfloor\frac{s}{L\tau}\rfloor L\tau \big) \right|\\
	+&\left|(\beta(s)-s)- \big(   \beta \big(  \lfloor\frac{s}{L\tau}\rfloor L\tau   \big)- \lfloor\frac{s}{L\tau}\rfloor L\tau \big)\right|\\
	\leq & \frac{\theta }{4}+\frac{\theta }{4}\left| \dfrac{ \big(   \alpha \big(  \lfloor\frac{s}{L\tau}\rfloor L\tau   \big)- \lfloor\frac{s}{L\tau}\rfloor L\tau \big)  }{\theta/4} -\dfrac{ \big(   \beta \big(\lfloor\frac{s}{L\tau}\rfloor L\tau   \big)- \lfloor\frac{s}{L\tau}\rfloor L\tau \big)  }{\theta/4}   \right|+\frac{\theta}{4}\\
	\leq & \theta.
	\end{split}
	\end{align*}
	Moreover, for any $s\in [0, t], $ we have $d(\phi_{\alpha(s)}x, \phi_{\beta(s)}x)<\dfrac{\eta}{3}$. Since for any $0\leq s\leq t,$ one has
	\begin{align*}
	\begin{split}
	d(\phi_s y , \phi_s z)\leq & d(\phi_s y, \phi_{\alpha(s)}x)+d(\phi_{\alpha(s)}x, \phi_{\beta(s)}x )+d(\phi_{\beta(s)}x , \phi_s z) \\
	\leq & \epsilon +\frac{\eta}{3}+ \epsilon <\eta.
	\end{split}
	\end{align*}
	Especially,  $d(\phi^i_{L\tau }y, \phi^i_{L\tau } z)\leq \eta, i=0, 1, \cdots, n-1.$
	For any $A\in \bigvee\limits_{i=0}^{n-1} \phi_{L\tau }^{-i} \xi$, there exist $i_0, i_1, \cdots, i_{n-1}\in \{1, 2, \cdots, q+1\}$ such that
	$$A=A_{i_0}\cap \phi^{-1}_{L\tau }(A_{i_2})\cap \cdots \phi^{-(n-1)}_{L\tau } (A_{i_{n-1}}).$$
	By the choice of $\eta_0$, for the fixed $S_\alpha$, there exist at most $2^n$ many $A' s$ such that
	$A\cap B(x, t, \epsilon, \phi )\neq \emptyset.$  Note that the first term of $S_{\alpha}$ is zero and two consecutive terms of it differ at most by 1.
	So the number of  $S_{\alpha }$ is  at most $3^{n-1}$.  Hence, $\# W_n\leq 6^n.$
	
	For $k=0,1,\cdot\cdot\cdot, K-1,$ let
	\begin{align*}
	D_{k,n}=\{x\in E_{k}: \mu(B(x,t,\epsilon,\phi))> 6^{2n}\exp(-(k\gamma-p)n)\},
	\end{align*}
	and  
	\begin{align*}
	D_{\infty,n}=\{x\in E_{\infty}: \mu(B(x,t,\epsilon,\phi))> 6^{2n}\exp(-(\frac{1}{b}-p)n)\}.
	\end{align*}
	To prove  (\ref{1}), we consider the case for $k=0,1,\cdots,K-1$.
	
	If we can prove that $\sum_{n=N}^{\infty}\mu(D_{k,n})<\infty$, then apply the Borel-Cantelli Lemma: for $\mu$-a.e. $x\in E_{k}$,
	$$\liminf\limits_{n\rightarrow \infty }-\dfrac{\log \mu(B(x,nL\tau, \epsilon, \phi ))}{n}+ \frac{p+2\log6}{L\tau}\geq \frac{k\gamma}{L\tau}.$$
	Since $L \geq  \dfrac{p+2\log6}{p\tau},$
	we obtain that
	$$\liminf\limits_{n\rightarrow \infty }-\dfrac{\log \mu(B(x,nL\tau, \epsilon, \phi ))}{n}+ p \geq \frac{k\gamma}{L \tau}.$$
	Hence
	\begin{align*}
	\int_{X_{M}}  \liminf\limits_{n\rightarrow \infty }-\dfrac{\log \mu(B(x,nL\tau, \epsilon, \phi ))}{n} + p  ~d\mu
	&\geq \frac{\sum\limits_{k=0}^{K-1}k\gamma\mu(E_{k})}{L\tau}
	\\&= \frac{\sum\limits_{k=0}^{K-1}k\gamma\mu(B_{k})-\sum\limits_{k=0}^{K-1}k\gamma(\mu(B_{k})-\mu(E_{k}))}{L\tau}
	\\& \geq \frac{1}{L \tau}\int_{h^{-1}([0, M))} h(y) d\pi(y) -\frac{\gamma+K(K-1)\gamma b}{L\tau}.
	\end{align*}
	Let $b \rightarrow 0$, and then let  $\gamma\rightarrow 0$ (by letting $K$ tend to infinity), and we have
	\begin{align*}
	\int_{X_{M}}p+ \underline{h}_{\mu}(\phi, x)d\mu \geq \frac{1}{L \tau}\int_{h^{-1}([0, M))}h(y)d\pi(y).
	\end{align*}
	Now we estimate the measures of $D_{k,n}.$
	For any $x\in D_{k, n}$, in those  $6^{n}$ atoms of $\bigvee_{i=0}^{n-1}\phi_{L \tau }^{-i}\xi $
	such that $ A \cap B(x, t, \epsilon, \phi),$ there exists at least one corresponding atom of $\bigvee_{i=0}^{n-1}\phi_{L \tau }^{-i} \xi $  whose measure is greater than $6^{n}\exp(-(k\gamma-p)n)$. The total number of such atoms will not exceed $6^{-n}\exp((k\gamma-p)n).$ Hence $Q_{k, n},$ the total number of elements of $\bigvee_{i=0}^{n-1}\phi_{L \tau }^{-i}\xi$ that intersect $ D_{k,n},$ satisfies
	$$ Q_{k,n}\leq   6^{n}6^{-n}\exp((k\gamma-p)n)=\exp((k\gamma-p)n).$$
	Let $S_{k,n}$ denote the total measure of such $Q_{k,n}$ elements of $\bigvee_{i=0}^{n-1}\phi_{L \tau }^{-i}\xi $ whose intersections with $E_{k}$ have positive measure. By the definition of $E_{k}$,
	$$S_{k, n}\leq Q_{k,n} \exp((-k\gamma+ b)n)\leq \exp((b-p)n), $$
	which follows that
	$$ \mu(D_{k,n})\leq S_{k,n} \leq\exp((b-p) n).$$
	Since $b< p$, we have $\sum_{n=N}^{\infty}\mu(D_{k,n})\leq \infty $.
	
	To prove $(\ref{2})$, we need estimate the measures of $D_{\infty,n}$. In the above treatment for $D_{k,n}$, replacing  $k\gamma$ (resp.$D_{k,n}, Q_{k,n} $ and $S_{k,n}$) by $\frac{1}{b}$ (resp.$D_{\infty,n}, Q_{\infty,n}, $ and $S_{\infty,n}$), it also holds that $\sum\limits_{n=N}^{\infty}D_{\infty, n} < \infty.$ Then apply  the Borel-Cantelli Lemma again: for $\mu$-a.e. $x\in E_{\infty}$,
	$$  \liminf\limits_{n\rightarrow \infty} -\frac{1}{nL\tau}\log \mu(B(x,nL\tau,\epsilon,\phi))+ p \geq \frac{1}{b}.$$
	Let $b$ go to $0$, we conclude that  for $\mu$ almost every $x\in X_{\infty},$
	$$ \lim\limits_{\epsilon \rightarrow 0} \liminf\limits_{n\rightarrow \infty} -\frac{1}{nL\tau}\log \mu(B(x,nL\tau,\epsilon,\phi))+p= \infty.$$
\end{proof}

\section{Proof of Theorem \ref{main2}}\label{four}



The following theorem  will be used in proving the Theorem \ref{main2}.

\begin{theorem}\rm\cite{Dou}\label{dou}
	Let $(X,\phi)$ be a compact metric flow without fixed points. If $K$ is a non-empty compact subsets of $X$, then
	\begin{equation*}
	h_{top}^{B}(\phi,K)=\sup\{{\underline{h}_{\mu}(\phi):\mu\in \mathcal{M}(X),\mu(K)=1}\}
	\end{equation*}
\end{theorem}

\begin{proof}[Proof of Theorem \ref{main2}]
	 Let $\mu \in \mathcal{M}_{\phi}(X)$ and $Y$ a subset of $X$ with $\mu(Y)=1$.  Let $\{Y_{n}\}_{n\in \mathbb{N}}$ be an increasing sequence of compact subsets of $Y$ such that $\mu(Y_{n})>1-\frac{1}{n}$ for each $n\in \mathbb{N}$. Then by Proposition \ref{w},
	\begin{align}\label{4.1}
	h_{top}^{B}(\phi, Y)\geq h_{top}^{B}(\phi,\bigcup_{n\in \mathbb{N}}Y_{n})=\sup_{n}h_{top}^{B}(\phi,Y_{n})=\lim_{n\rightarrow\infty}h_{top}^{B}(\phi, Y_{n}).
	\end{align}
	Denote by $\mu_{n}$ the restriction of $\mu$ on $Y_{n}$, i.e., for any $\mu$-measurable set $A\subset X $,
	\begin{equation*}
	\mu_{n}(A)=\frac{\mu(A\bigcap Y_{n})}{\mu(Y_{n})}.
	\end{equation*}
	Joint with Theorem \ref{dou}, we have
	 \begin{equation}\label{4.2}
	 h_{top}^{B}(\phi,Y_{n})=\sup\{\underline{h}_{\nu}(\phi) : \nu\in M(X),\nu(Y_{n})=1\}
	 \geq \underline{h}_{\mu_{n}}(\phi).
	 \end{equation}
Note that
\begin{align*}
\underline{h}_{\mu_{n}}(\phi)&=\int_{Y_{n}}\lim_{\epsilon\rightarrow 0}\liminf_{t\rightarrow \infty}-\frac{1}{t}\log\mu_{n}(B(x,t,\epsilon,\phi))d\mu_{n}\\
&=\frac{1}{\mu(Y_{n})}\int_{Y_{n}}\lim_{\epsilon\rightarrow 0}\liminf_{t\rightarrow \infty}-\frac{1}{t}\log\frac{\mu(B(x,t,\epsilon,\phi)\cap Y_{n})}{\mu(Y_{n})}d\mu\\
&\geq \frac{1}{\mu(Y_{n})}\int_{Y_{n}}\lim_{\epsilon\rightarrow 0}\liminf_{t\rightarrow \infty}-\frac{1}{t}\log\frac{\mu(B(x,t,\epsilon,\phi))}{\mu(Y_{n})}d\mu\\
&=\frac{1}{\mu(Y_{n})}\int_{Y_{n}}\lim_{\epsilon\rightarrow 0}\liminf_{t\rightarrow \infty}-\frac{1}{t}\log\mu(B(x,t,\epsilon,\phi))d\mu.
\end{align*}
By Theorem \ref{ka},
$$ \int_{Y}\lim_{\epsilon \rightarrow 0}\liminf_{t\rightarrow \infty}-\frac{1}{t}\log\mu(B(x,t,\epsilon,\phi))d\mu=\underline{h}_{\mu}(\phi)=h_{\mu}(\phi_{1}).$$
Hence $$\lim\limits_{n \rightarrow \infty}\underline{h}_{\mu_{n}}(\phi) \geq h_{\mu}(\phi_{1}).$$
Together with (\ref{4.1}) and (\ref{4.2}),
\begin{equation*}
 h_{top}^{B}(\phi,Y)\geq h_{\mu}(\phi_{1}).
\end{equation*}

\end{proof}
 Since $\mu(G_{\mu}(\phi))=1$ for $\mu \in \mathcal{E}_{\phi},$  we have  the following corollary by Theorem \ref{main2} .
\begin{corollary}\label{lower}
	 Let $( X, \phi)$ be a compact metric flow without fixed points and $\mu \in \mathcal {E}_{\phi}(X)$, then $h_{\mu}(\phi_{1}) \leq h_{top}^{B}(\phi, G_{\mu}(\phi)).$
\end{corollary}

\section{Proof of Theorem \ref{main}}\label{five}
In this section, we will show the proof of  Theorem \ref{main}. Corollary \ref{lower} gives the lower bound. For the upper bound, we use the ideas of Pfister and Sullivan \cite{Pfi}.

 For $\mu \in  \mathcal{E}_{\phi}(X),$ let $\left\lbrace K_{m}\right\rbrace _{m\in \mathbb{N}}$ be a decreasing sequence of closed convex neighborhoods of $\mu$ in $\mathcal{M}(X)$ and set
 $$ A_{n,m}=\left\lbrace {x\in X: \frac{1}{n}\int_{0}^{n}\delta_{x}\circ\phi_{-\tau}d\tau \in K_{m}}\right\rbrace, ~~\text {for} ~~m, n \in \mathbb{N}.$$
Then for any $m$, $N \geq 1$, $G_{\mu}(\phi)\subset \bigcup\limits_{n \geq N} A_{n,m}$.
For $E \subset X $ and $\epsilon>0,$ we say that $E$ is a $(t,\epsilon)$-strongly separated set in $X$ if for every $x, y \in E, x \neq y$ and for every  $\alpha, \beta \in Rep [0,t]$,
\begin{align*}
d(\phi_{\alpha(s)}x, \phi_{s}y)> \epsilon ~~~ \text {for~~ some}~ s\in [0,t]
\end{align*}
or
\begin{align*}
d(\phi_{\beta(s)}y,\phi_{s}x) >\epsilon
~~~ \text {for~~ some }~ s\in [0,t].
\end{align*}
Denote by $S_{n}(A_{n,m}, \epsilon)$ the maximal cardinality of any $( n, \epsilon)$-strongly separated subset of $A_{n,m}.$

\noindent {\bf Claim 2}. It holds that 
$$ \lim\limits_{\epsilon \rightarrow 0} \lim\limits_{m \rightarrow \infty}\limsup_{n\to\infty}\frac{1}{n}\log S_{n}(A_{n,m},\epsilon)\leq h_{\mu}(\phi_{1}).$$
\begin{proof}[Proof of the Claim 2]
	If not, suppose that
	$$\lim\limits_{\epsilon \rightarrow 0} \lim\limits_{m \rightarrow \infty}\limsup_{n\to\infty}\frac{1}{n}\log S_{n}(A_{n,m},\epsilon)> h_{\mu}(\phi_{1})+ \delta$$
for some $\delta >0.$ Then there exist $\epsilon_{0}>0$ and $M\in \mathbb{N}$ such that for any $0<\epsilon< \epsilon_{0}$ and any $m > M$, it holds that
\begin{align*}
\limsup_{n\to\infty}\frac{1}{n
}\log S_{n}(A_{n,m},\epsilon)> h_{\mu}(\phi_{1}) + \delta.
\end{align*}
Hence, we can find a sequence $\left\lbrace m(n) \right\rbrace $ with $m(n) \rightarrow\infty $ such that
\begin{align*}
\limsup_{n\to\infty}\frac{1}{n}\log S_{n}(A_{n,m(n)},\epsilon)\geq h_{\mu}(\phi_{1}) + \delta.
\end{align*}
Now let $E_{n}$ be a $(n , \epsilon)$-strongly separated set of $A_{n,m(n)} $ with maximal cardinality and define
\begin{align*}
\delta_{n}=\frac{1}{\sharp E_{n}}\sum\limits_{x\in E_{n} } \delta_{x} ~~ \text{and} ~~\mu_{n}=\frac{1}{n}\int_{0}^{n} \delta_{n}\circ\phi_{-\tau} d\tau.
\end{align*}
Since
\begin{align*}
\frac{1}{n}\int_{0}^{n} \delta_{x}\circ\phi_{-\tau} d\tau\in K_{m(n)}, ~\text{for}~~\text{any}~ x\in E_{n}
\end{align*}
and
\begin{align*}
\mu_{n}=\frac{1}{\sharp E_{n}}\sum\limits_{x\in E_{n}}\frac{1}{n}\int_{0}^{n}\delta_{x}\circ\phi_{-\tau} d\tau,
\end{align*}
 $\mu_{n} \in K_{m(n)}$ by the convexity of $K_{m}'s$. Hence $\mu_{n} \rightarrow  \mu $ as $n\rightarrow \infty$. For the above $\epsilon>0$,  there exists $0<\gamma<\epsilon$ such that for  any $0\leq s\leq 1$ and $x, y \in X$, we have  $d(\phi_{s}x,\phi_{s}y)< \epsilon$ whenever $d(x, y)< \gamma$.
Let $\beta$ be a finite Borel partition of $X$ such that diam$(\beta) < \gamma $ and $\mu(\partial\beta)=0.$ Then  each element  of  $\bigvee_{i=0}^{n}\phi_{-i}\beta$  contains at most one point in $E_{n}.$ In fact,  for any $x\neq y \in E_{n}$, there exists  $s\in[0,n]$ such that  $d(\phi_{s}x, \phi_{s}y)> \epsilon > \gamma.$ If $s\in \{ 0,1, \cdots, n\}$, then  $x, y  $ can't contain the same element of  $\bigvee_{i=0}^{n}\phi_{-i}\beta$.  If  $\lfloor s\rfloor\leq s\leq \lfloor s \rfloor+1 $, by the choice of $\gamma$, it follows that  $d(\phi_{\lfloor s\rfloor}x,\phi_{\lfloor s \rfloor}y)>\gamma $. Since  diam$(\beta)< \gamma $, this implies  that $ x, y \in E_{n}$  contain the different elements of  $\bigvee_{i=0}^{n}\phi_{-i}\beta$. 
Hence
$S_{n}(A_{n,m(n)},\epsilon)$ members of $\bigvee\limits_{i=0}^{n}\phi_{-i} \beta$ each have $\delta_{n}$-measure $S_{n}(A_{n,m(n)},\epsilon)$ and the others have $\delta_{n}$-measure zero. Fix natural numbers $q$, $n$ with $ 1<q<n$ and define $a(j)$ by $a(j)=[\frac{n+1-j}{q}]$ for $0\leq j\leq q-1$. Fix $0\leq j \leq q-1$. We have
\begin{align*}
\bigvee_{i=0}^{n}\phi_{-i} \beta=\bigvee_{r=0}^{a(j)-1}\phi_{-(rq+j)}\bigvee_{i=0}^{q-1} \phi_{-i}\beta \vee \bigvee_{l\in S}\phi_{-l} \beta
\end{align*}
and $S$ has cardinality at most $2q$. Therefore
\begin{align*}
\log S_{n}(A_{n,m(n)},\epsilon)=H_{\delta_{n}}(\bigvee_{i=0}^{n}\phi_{-i}\beta).
\end{align*}
Given $0<\theta<1$. Similar to  the above discussion, 
each element  of  $\bigvee_{i=0}^{n}\phi_{-i-\theta}\beta$  contains at most one point in $E_{n},$ then
\begin{align*}
&\log S_{n}(A_{n,m(n)},\epsilon)=H_{\delta_{n}}(\bigvee_{i=0}^{n}\phi_{-i-\theta}\beta) \\&\leq \sum_{r=0}^{a(j)-1}H_{\delta_{n}\circ \phi_{-\theta}}(\phi_{-(rq+j)}(\bigvee_{i=0}^{q-1}\phi_{-i}\beta))+\sum\limits_{k\in S}H_{\delta_{n}\circ \phi_{ -\theta}}(\phi_{-k}\beta)\\&\leq\sum_{r=0}^{a(j)-1}H_{\delta_{n}\circ \phi_{-(\theta +rq+j)}}(\bigvee_{i=0}^{q-1}\phi_{-i}\beta)+2q\log \# \beta.
\end{align*}
It is easily seen that
\begin{align*}
\int_{0}^{1}\log S_{n}(A_{n,m(n)},\epsilon) d \theta &\leq \sum_{r=0}^{a(j)-1}\int_{0}^{1}H_{\delta_{n}\circ \phi_{-(\theta +rq+j)}}(\bigvee_{i=0}^{q-1}\phi_{-i}\beta)d\theta+2q\log \# \beta \\&\leq \sum_{r=0}^{a(j)-1}H_{\int_{0}^{1}\delta_{n}\circ \phi_{-(\theta +rq+j)} d\theta}(\bigvee_{i=0}^{q-1}\phi_{-i}\beta)+2q\log \#\beta.
\end{align*}
The last equality  follows from the concavity  of the function $-x\log x$.
Sum this inequality over $j$ from to $q-1$ , we can get
\begin{align*}
q\log S_{n}(A_{n,m(n)},\epsilon)=&\int_{0}^{1}q\log S_{n}(A_{n,m(n)},\epsilon) d \theta \\&\leq \sum_{p=0}^{n}H_{\int_{0}^{1} \delta_{n}\circ \phi_{-(p+\theta)}d\theta} (\bigvee_{i=0}^{q-1}\phi_{-i}\beta)+ 2q^{2}\log \# \beta.
\end{align*}
 If we divide by $n+1$, we obtain
\begin{align}
\frac{q}{n+1}\log S_{n}(A_{n,m(n)},\epsilon)& \leq \frac{1}{n+1}\sum_{p=0}^{n}H_{\int_{0}^{1} \delta_{n}\circ \phi_{-(p+\theta)}d\theta} (\bigvee_{i=0}^{q-1}\phi_{-i}\beta)+ \frac{2q^{2}\log \# \beta}{n+1}\notag
\\&\leq H_{\frac{1}{n+1}\sum\limits_{p=0}^{n}\int_{0}^{1} \delta_{n}\circ \phi_{-(p+\theta)}d\theta} (\bigvee_{i=0}^{q-1}\phi_{-i}\beta)+ \frac{2q^{2}\log \# \beta}{n+1}\notag\\&= H_{\frac{1}{n+1}\int_{0}^{n}\delta_{n}\circ \phi_{-\tau}d\tau} (\bigvee_{i=0}^{q-1}\phi_{-i}\beta)+ \frac{2q^{2}\log \# \beta}{n+1}.\label{*}
\end{align}
We know the members of $\bigvee\limits_{i=0}^{q-1}\phi_{-i}\beta$ have boundaries of $\mu$-measure zero, so $\lim\limits_{j\rightarrow \infty}\mu_{n_{j}}(B)=\mu(B)$ for each member $B$ of $\bigvee\limits_{i=0}^{q-1}\phi_{-i}\beta$ and  $\lim\limits_{j\rightarrow \infty}H_{\mu_{n_{j}}}(\bigvee\limits_{i=0}^{q-1}\phi_{-i}\beta)=H_{\mu}(\bigvee\limits_{i=0}^{q-1}\phi_{-i}\beta).$
Therefore replacing $n$ by $n_{j}$ in (\ref{*}) and letting $j$ go to $\infty$, we have
 $$\limsup _{n\to\infty}\frac{1}{n}\log S_{n}(A_{n,m(n)},\epsilon)\leq \frac{1}{q}H_{\mu}(\bigvee\limits_{i=0}^{q-1}\phi_{-i}\beta).$$
This lead to
$$ \limsup _{n\to\infty}\frac{1}{n}\log S_{n}(A_{n,m(n)},\epsilon)\leq h_{\mu}(\phi_{1}),$$
giving a contradiction.
\end{proof}
By the Claim 2, for each $\delta >0,$ there exists $\epsilon_{0}>0$ satisfying that for any $0<\epsilon<\epsilon_{0},$ there exists $M\in \mathbb{N}$ such that whenever $m>M$, it holds that
\begin{align*}
\limsup_{n\to\infty}\frac{1}{n}\log S_{n}(A_{n,m(n)}
,\epsilon)\leq h_{\mu}(\phi_{1})+ \frac{\delta}{2}.
\end{align*}
Let $E_{n,m}$ be a $(n,\epsilon)$-strongly separated set of $A_{n,m}$ with maximal cardinality, according to Lemma 1.3 of \cite{Thoma1}, it follows that 

 $$A_{n,m}\subset\bigcup\limits_{x\in E_{n,m}}B(x,n,2\epsilon,\phi).$$ Hence for
$s=h_{\mu}(\phi_{1})+ 2\delta,$ we have
\begin{align*}
\mathcal{M}_{N,2\epsilon}^{s}(\phi, G_{\mu}(\phi))&\leq \mathcal{M}_{N,2\epsilon}^{s}(\phi,\bigcup_{n \geq N}A_{n,m})\\&\leq \sum_{n \geq N}\mathcal{M}_{n, 2\epsilon}^{s}(\phi,A_{n,m})\\&\leq \sum_{n \geq N} \exp((h_{\mu}(\phi_{1})+\delta -s)n)
\\&=\sum_{n\geq N} \exp(-\delta n).
\end{align*}
Letting $n\rightarrow \infty$, we obtain 
$$\mathcal{M}_{2\epsilon}^{s}(\phi,G_{\mu}(\phi)) \leq \lim\limits_{N\rightarrow \infty}\sum_{n\geq N} \exp(-\delta n)=0,$$
which implies that $h_{top}^{B}(\phi, G_{\mu}(\phi))\leq h_{\mu}(\phi_{1}).$
\section{A Billingsley type Theorem for Bowen entropy}{}\label{six}

\begin{lemma}{\rm\cite{Dou}}\label{r}
	Let $(X,\phi)$ be a compact metric flow without fixed points. Let $0<\eta<1,$  $t>1$, and $\theta$ be as in Lemma \ref{imp}. Write $\tilde{t}=(1-\eta)t.$ Then for any $0<\epsilon<\theta$, 
	\begin{enumerate}[1.]	
		\item  if $y\in B(x,t, \epsilon, \phi)$, then $x\in B(y, \tilde{t},\epsilon,\phi);$
		
		\item  if $y \in B(x,t, \frac{\epsilon}{2}, \phi)$, then $B(x,t, \frac{\epsilon}{2}, \phi)\subseteq B(y, \tilde{t}, \epsilon, \phi)$.
	\end{enumerate}
	
\end{lemma}

\begin{lemma}{\rm \cite{Dou}}\label{Dou}
	Let	$(X, \phi)$  be a compact metric flow without fixed points. For $0<\eta<1$, let $\theta$ be as Lemma \ref{imp}. Let $\mathcal{B}=\{B(x, t, r, \phi)\}_{(x, t)\in I}$  be a family of reparametrization balls in $X$
	with $0<r< \frac{\theta}{2}$ and $ t> \frac{1}{(1-\eta)^{2}}$. Then there exists a finite or countable subfamily $\mathcal{B}'=\{B(x, t, r, \phi )\}_{(x, t)\in I'}(I'\subset I)$ of pairwise disjoint reparametrization balls in $\mathcal{B}$ such that $$ \bigcup_{B\in \mathcal{B}} B\subset \bigcup_{ (x, t)\in I'}B(x, \hat{t}, 5r, \phi)$$ 
	where $\hat{t}=(1-\eta)^2t$.
	\end{lemma}
Now we proceed with the proof  of Theorem \ref{main3}. We follow the proof originally due to \cite{Ma}.

\begin{proof}[Proof of Theorem \ref{main3}]
	
	(1) For $0<\eta <1$, let $\theta>0$ be as in  Lemma \ref{imp}.  Fix $\epsilon >0$,  since $\underline{h}_{\mu}(x)\leq s$ for all $x\in E$, we have  $E=\bigcup_{k=1}^\infty E_k$,
	where $$E_k=\left\{ x\in E:~ \liminf\limits_{t\to \infty } \dfrac{-\log \mu (B(x, t, r, \phi ))}{t} < s+\epsilon ~~\text {for}~ \text {all}~~ r\in (0, \frac{1}{k})
	\right\}.$$
	Now  we fix $k\geq 1$ and  $0<r< \min\left\lbrace  { \frac{1}{5k}, \frac{\theta}{2}}\right\rbrace $.
	For each $x\in E_k$, there exists a strictly  increasing sequence $\{t_j(x)\}_{j=1}^{\infty}$ such that $\mu (B(x, t_j(x), r, \phi))\geq e^{-(s+\epsilon )t_j(x)}$ for all $j\geq 1.$ So, for any $N\geq 1$,
	the set $E_k$ is contained  in the  union  of the sets in  the family $\mathcal{F}=\{B(x, t_j(x), r, \phi): x\in E_k, t_j(x)\geq N\}$.
	By Lemma  \ref{Dou}, there exists  a finite or countable  subfamily $\mathcal{G}=\{B(x_i, t_i, r, \phi)\}_{i\in I}\subset \mathcal{F}$ consisting of disjoint balls such that，
	$$E_k \subseteq \bigcup_{i\in I}B(x_i, \hat{t}_i, 5r, \phi). $$
	where $\hat{t}_i=(1-\eta)^2t_i\geq (1-\eta)^{2}N, \mu (B(x_{i}, t_i(x), r, \phi))\geq e^{-(s+\epsilon )t_i(x)}$ for all $i \in I .$ Therefore,
	
	$$\mathcal{M}_{(1-\eta)^{2}N, 5r }^{\frac{s+\epsilon}{(1-\eta)^{2}}}(\phi, E_{k})\leq   \sum_{i\in I} e^{-\frac{s+\epsilon}{(1-\eta)^{2}}\hat{t}_i} =\sum_{i\in I} e^{-(s+\epsilon )t_i}\leq \sum_{i\in I} \mu (B(x_i, t_i, r, \phi))\leq 1,$$
	where the disjointness of $\{B(x_i, t_i, r, \phi)\}$ is used in the last inequality. It follows that 
	$$\mathcal{M}_{ 5r }^{\frac{s+\epsilon}{(1-\eta)^{2}}}(\phi, E_k)=\lim\limits_{N\to \infty }\mathcal{M}_{(1-\eta)^{2}N, 5r }^{\frac{s+\epsilon}{(1-\eta)^{2}}}(\phi, E_k)\leq 1,$$
	which implies that 
	$\mathcal{M}_{ 5r }^{\frac{s+\epsilon}{(1-\eta)^{2}}}(\phi, E_k)\leq 1$  for all $0< r < \frac{1}{5k}$.
	
	 Letting $r\to 0$  yields 
	$\mathcal{M}^{\frac{s+\epsilon}{(1-\eta)^{2}}}(\phi, E_k)< 1$, hence 
	 $h_{top}^{B}(\phi, E_k)\leq  \frac {s+\epsilon}{(1-\eta)^{2}}$   for any $k \geq 1$.
	Since the Bowen entropy is countable stable(see Proposition \ref{w}), it follows that 
	
	$$h_{top}^{B}(\phi, E)=h_{top}^{B}(\phi, \bigcup_{k=1}^{\infty} E_k)=\sup_{k\geq 1}h_{top}^{B}(\phi, E_k)\leq \frac {s+\epsilon}{(1-\eta)^{2}}.$$
	Let $\eta \rightarrow 0$, we have $h_{top}^{B}(\phi, E)\leq s+\epsilon.$ Therefore, $h_{top}^{B}(\phi, E) \leq s$ as $\epsilon$ is arbitrary.
	
	(2) For $0<\eta <1$, let $\theta>0$ be as in  Lemma \ref{imp}. Let us first fix  an $\epsilon>0$. For each $k\geq 1$, put 
	$$E_k=\left\{ x\in E:~ \liminf\limits_{t\to \infty } \dfrac{-\log \mu (B(x, t, r, \phi ))}{t} > s-\epsilon 
~\text{for~all}~ r\in (0, \dfrac{1}{k}) 	\right\}.$$
	
	Since $\underline{h}_{\mu}(x)\geq s$ for all $x\in E$, the sequence $\{E_k\}_{k=1}^{\infty}$ increases to $E$.
	So by the continuity of the measure, we have 
	$$\lim\limits_{k\to \infty } \mu (E_k)=\mu(E)>0.$$
	Then fix some $k\geq 1$
	with $\mu(E_k)>\frac{1}{2}\mu(E)$. For each $N\geq 1$, put
	$$E_{k, N}=\left\{ x\in E_k:~  \dfrac{-\log \mu (B(x, t, r, \phi ))}{t} > s-\epsilon ~\text{for all }~ t\geq N 
 ~\text{and}~ r\in (0, \dfrac{1}{k})	\right\}.$$
	Since the sequence $\{E_{k, N}\}_{N=1}^{\infty }$ increases to $E_k$, we may pick an $N'\geq 1$ such that $\mu(E_{k, N'})>\frac{1}{2}\mu(E_k)$.
	Write  $E'=E_{k, N'}$ and $r'=\frac{1}{k}$, then $\mu (E')>0$.
	Since $B(x,t, r, \phi)\subseteq B(x,t, \dfrac{1}{k}, \phi) $ for $0<r\leq \dfrac{1}{k}$, we obtain that  
	\begin{align}\label{eqq}
	\mu(B(x, t, r, \phi ))\leq e^{-(s-\epsilon )t}.
	\end{align}
	for all $x\in E', 0<r\leq \min\left\lbrace {r', \theta}\right\rbrace , t\geq N'$.
	
	Now suppose that  $\mathcal{F}=\{ B(y_i, t_i, \frac{r}{2}, \phi)\}_{i\geq 1}$ be  a countable or finite family  so that $y_{i}\in X,$ $\widetilde{t_{i}}=(1-\eta)t_{i} \geq N'$ and  $\bigcup\limits_{i} B(y_i, t_i, \frac{r}{2}, \phi) \supset E'$. We may assume that for each $i$, $B(y_i, t_i, \frac{r}{2}, \phi)\bigcap E' \neq \emptyset $
	for all $ i\geq 1, 0<r\leq \min \left\lbrace r', \theta \right\rbrace $.
	
	For each  $i\geq 1$, we can choose    $x_i\in E'\cap B(y_i, t_i, \frac{r}{2}, \phi)$.
    By  Lemma \ref{r}, we have  
	$$B(y_i, t_i, \frac{r}{2}, \phi)\subset B(x_i, \widetilde{t_i},  r , \phi).$$
In combination  with  (\ref{eqq}), this implies 
	\begin{align*}
	\sum_{i\geq 1} e^{-(1-\eta)t_i(s-\epsilon )}&= \sum_{i\geq 1} e^{-\widetilde{t_{i}}(s-\epsilon)}\\&
	 \geq \sum_{i\geq 1} \mu (B(x_i, \widetilde {t_i}, r, \phi))\geq \sum_{i\geq 1} \mu (B(y_i,  {t_i}, \frac{r}{2}, \phi)) \geq \mu (E').
	\end{align*}
	It follows that 
	$\mathcal{M}^{(1-\eta) (s-\epsilon )}(\phi, E')\geq \mathcal{M}^{(1-\eta) (s-\epsilon )}_{\frac{N}{1-\eta}, \frac{r}{2}}(\phi, E') \geq \mu (E')>0.$
	
	Therefore $h_{top}^B(\phi, E')\geq (1-\eta)(s-\epsilon) $.
	Let $\eta \rightarrow 0$, we then have 
	$h_{top}^B(\phi, E')\geq s-\epsilon $.
	Hence $h_{top}^B(\phi, E')\geq s$  as $\epsilon $ is arbirtary small. Hence $h_{top}^B(\phi, E)\geq s$. This completes the proof of the theorem.

\end{proof}

\end{document}